\documentclass[reqno]{amsart}

\usepackage{titlesec}
\titleformat{\section}
  {\normalfont\large\bfseries}
  {\thesection}{1em}{}

\titlespacing*{\section}
  {0pt}{3ex plus 1ex minus .2ex}{2ex}
\titleformat{\subsection}
  {\normalfont\normalsize\bfseries}
  {\thesubsection}{1em}{}

\titlespacing*{\subsection}
  {0pt}{2ex plus .5ex}{1.5ex}

\usepackage[T1]{fontenc}
\usepackage{amsmath,amssymb, amsopn, amsthm}
\usepackage{enumitem, framed, xcolor}
\usepackage{nicefrac}

\usepackage{hyperref}
\hypersetup{hidelinks}
\hypersetup{hidelinks,hypertexnames=false}

\usepackage{tikz-cd}
\usetikzlibrary{graphs, graphs.standard}

\usepackage{lmodern}
\usepackage[protrusion=true]{microtype}

\usepackage{xcolor}
\usepackage[normalem]{ulem}

\theoremstyle{definition} 
\newtheorem{theorem}{Theorem}[section]
\newtheorem{lemma}[theorem]{Lemma}

\newtheorem{corollary}[theorem]{Corollary}

\newtheorem{fact}[theorem]{Fact}

\newtheorem{example}[theorem]{Example}

\newtheorem{problem}[theorem]{Problem}

\newtheorem{definition}[theorem]{Definition}

\renewenvironment{proof}[1][\unskip]{%
\par
\noindent
\textbf{Proof #1.}
\noindent}
{\hfill$\blacksquare$

\bigskip}

\usepackage{etoolbox}
\newcommand{\defqedsymbol}{\hfill\ensuremath{\square}} 
\AtBeginEnvironment{definition}{\pushQED{\defqedsymbol}}
\AtEndEnvironment{definition}{\popQED}

\let\<=\langle
\let\>=\rangle

\def\land{\wedge}       
\def\lor{\vee}

\let\models=\vDash

\renewcommand{\setminus}{\smallsetminus}
\def\mindot{\buildrel\textstyle.\over
    {\hbox{\vrule height.55ex width0pt
        \smash{\hbox{\mathsurround=0pt$-$}}}}}

\def\defeq{\mindot}

\makeatletter
\newcommand\RedeclareMathOperator{%
  \@ifstar{\def\rmo@s{m}\rmo@redeclare}{\def\rmo@s{o}\rmo@redeclare}%
}
\newcommand\rmo@redeclare[2]{%
  \begingroup \escapechar\m@ne\xdef\@gtempa{{\string#1}}\endgroup
  \expandafter\@ifundefined\@gtempa
     {\@latex@error{\noexpand#1undefined}\@ehc}%
     \relax
  \expandafter\rmo@declmathop\rmo@s{#1}{#2}}
\newcommand\rmo@declmathop[3]{%
  \DeclareRobustCommand{#2}{\qopname\newmcodes@#1{#3}}%
}
\@onlypreamble\RedeclareMathOperator
\makeatother

\title{More on modal logics and deduction}

\author[Gyenis, Zal\'an]{Zal\'an Gyenis}
\address{%
	Jagiellonian University \\
	Grodzka 52 \\
	31-007 Krak\'ow \\
	Poland}
\email{zalan.gyenis@uj.edu.pl}

\author[Moln\'ar, Zal\'an]{Zal\'an Moln\'ar}
\address{%
	Eötvös Loránd University\\
	Muzeum krt.4a \\
	1088 Budapest \\
	Hungary}
\email{mozaag@gmail.com}

\author[\"Ozt\"urk, \"Ovge]{\"Ovge \"Ozt\"urk}
\address{%
    Yozgat Bozok University\\
	Capanoglu Mah. Cemil Cicek Cad. No: 217/1\\
	66100 Yozgat\\
	Turkey
}\email{ovgeovge@gmail.com}

\subjclass{Primary: 06E25, 03B45. Secondary: 03G10, 03G27.}
\keywords{Boolean frames, Modal algebras, Congruential modal logics, Congruence extension property, (Local) deduction detachment theorem}
\date{\today}

\def\gA{\mathfrak{A}}

\def\gB{\mathfrak{B}}
\def\gC{\mathfrak{C}}
\def\gX{\mathfrak{X}}
\def\gD{\mathfrak{D}}

\def\NVar{\mathsf{L}_{f(0)=0}}
\def\TVar{\mathsf{L}_{f(1)=1}}
\def\Sna{\mathsf{L}_{sna}}

\def\Con{\operatorname{Con}}

\def\Bf{\mathsf{Bf}}
\def\AND{\sqcap}
\def\OR{\sqcup}
\def\bbN{\mathbb{N}}
\def\cP{\wp}
\def\CEP{\mathsf{CEP}}
\def\0{0}
\def\1{1}
\def\f{f}
\def\D{\Diamond}
\def\B{\Box}

\def\bfK{\mathbf{K}}

\let\emptyset=\varnothing

\usepackage{cleveref}

\begin{document}

	\begin{abstract}
		We study the relation between additivity and deduction theorems in the algebraic semantics of congruential modal logic. Additivity of the modal operator is well-known to imply the local deduction-detachment theorem. Our main theme is that deduction properties of modal logic persist far beyond the additive setting. We introduce the notion of a strongly non-additive variety, and then we prove that there are continuum many strongly non-additive minimal discriminator varieties of Boolean frames; equivalently, continuum many strongly non-additive maximal congruential modal logics with deduction-detachment theorem. Moreover, every normal modal logic can be transformed, in an injective way, into a strongly non-additive one while preserving the (local) deduction theorem. Finally, we show that neither the class of congruential modal logics with the local deduction theorem nor its complement is elementary.
	\end{abstract}
\maketitle

\section{The theme}

A \emph{Boolean frame} ($\Bf$, for short) is a structure $\gA = \<A, \AND, -, \0, \f\>$, where $\<A$, $\AND$, $-$, $\0\>$ 
is a Boolean algebra and $\f:A\to A$ is an arbitrary unary function.\footnote{We make use of other standard Boolean operations such as $\OR$, $\to$, $\1$. These are defined in the usual way. In some parts of the literature a modal algebra is a Boolean algebra with an extra \emph{operator} $\D$ which is \emph{normal} and \emph{additive} (that is, $\D\0=\0$ and $\D(a\OR b) = \D a\OR \D b$; or, using the dual operator $\B$ we can write $\B\1=\1$ and $\B(a\AND b)=\B a\AND\B b$; see e.g. \cite[Def. 10.1]{Ono2019}). In some other parts of the literature, e.g. in \cite{Krawczyk2023}, a modal algebra
is just a Boolean algebra with an arbitrary extra unary function. To resolve this conflict we follow \cite{Hansson1973} and use the expression Boolean frame to refer to modal algebras in the latter sense.} The class $\Bf$ is an equational class (variety) defined by the Boolean equations. It is known that the class $\Bf$ gives an equivalent algebraic semantics
to what is called congruential or classical modal logic \cite{GyenisMolnarNDJFL}.
As $\Bf$ is a variety, the lattice of axiomatic extensions of congruential modal logic is dually isomorphic to the lattice of subvarieties of $\Bf$ (see e.g. \cite{Krawczyk2023}).

Suppose that the variety $\mathbf{V}$ gives an equivalent algebraic semantics to the logic $\mathcal{L}$. Then 
by Czelakowski \cite[p.381]{Czelakowski1986}, $\mathbf{V}$ has the \emph{congruence extension property} (CEP) 
iff $\mathcal{L}$ has the \emph{local deduction-detachment theorem} (LDDT). Also, $\mathbf{V}$ has \emph{equationally definable principal congruences} (EDPC) iff the logic has the \emph{(global) deduction-detachment theorem} (DDT), see \cite{Font}.\footnote{We recall the relevant definitions later on. If the logic is algebraized by a quasi-variety, then there are similar transfer theorems, but in this paper we will not deal with such cases.}

It is known that even though the smallest \emph{normal} modal logic 
$\bfK$ does not have a deduction theorem \cite{Czelakowski1985}, it has a local deduction-detachment theorem 
(Perzanowski,\footnote{Czelakowski \cite[p.~380]{Czelakowski1986} credits this result to Jerzy Perzanowski. As we could not find a published work by Perzanowski establishing this result, we cite Czelakowski's paper as the source of the attribution.} cf. \cite[p.380]{Czelakowski1986}). Moreover,
every axiomatic extension of $\bfK$ has a LDDT, or equivalently, every \emph{normal and additive}\footnote{$f$ is normal if $f(0)=0$, and additive if $f(x\OR y)=f(x)\OR f(y)$.} subvariety of $\Bf$ has the CEP.\footnote{This statement is part of the folklore, mentioned by e.g. \cite[Proposition 1]{Kowalski2000}.}
In a recent paper, Krawczyk \cite{Krawczyk2023} proved that there
are continuum many axiomatic extensions of congruential modal logic that do not admit the LDDT, equivalently, there are continuum many subvarieties of $\Bf$ that lack the CEP (see \cite[Theorem 6.11]{Krawczyk2023}). This result was extended in \cite{GyenisMolnarNDJFL},
where it is proved that certain structural properties of the logic can further be assumed.
Additivity of the modal operator implies the LDDT, and thus the continuum of congruential modal logics without the LDDT constructed in \cite{Krawczyk2023,GyenisMolnarNDJFL} are necessarily non-additive.\\

In this paper we explore the question how much additivity and (local) de\-duc\-tion-de\-tach\-ment theorems are related. 
We take a closer look at \emph{non}-additive varieties \emph{with} the congruence extension property or \emph{with} EDPC.
There is an important caveat here when it comes to non-additive varieties. Consider, as possibly the easiest example,
an algebra (or variety) where $f(x) = -x$. This $f$ is non-additive (except for trivial cases) and CEP or EDPC are
obviously satisfied. Or take any additive $g$ and let $f(x)=-g(x)$ or $f(x)=-g(-x)$. Then $f$ typically\footnote{To get a 
non-additive $f$ we need to ensure that $g$ does not satisfy $g(x\OR y) = g(x)\AND g(y)$ in the first case, and $g(x\AND y) = g(x)\AND g(y)$ in the second.} is non-additive, but CEP or EDPC are transferred from $g$. This is, of course, the case of term-equivalence. 
Term-equivalent algebras share the same universe and the same clone of term operations: every basic operation in one algebra 
can be expressed as a composition of operations in the other, and vice versa. Term-\-e\-qui\-va\-lent algebras (varieties of them)
are generally considered to be the same for all intents and purposes.

Let us say that a variety of Boolean frames is equivalent to an additive variety of Boolean frames 
if for every algebra from the variety there is an 
additive operation with which it is term-equivalent (over the Boolean operations, of course). Note that this is a weaker notion than term-equivalence of varieties, as here the terms expressing one another are not uniform but can vary from algebra to algebra.  

\begin{definition} 
	A variety $\mathbf{V}$ of Boolean frames is called \emph{strongly non-additive} if it is not equivalent to 
	any additive variety of Boolean frames. 
\end{definition}

In this paper we prove that there are continuum many minimal strongly non-additive varieties with EDPC. Equivalently, 
continuum many strongly non-additive maximal modal logics with deduction-detachment theorem. Further, we show that
such varieties (logics) abound: every normal modal logic can be turned (in an injective way) into a strongly non-additive one keeping the (global/local) deduction-detachment theorem. 

We also show that neither the class of Boolean frames with the CEP, nor its complement can be axiomatized in first-order logic.\footnote{For most of the readers this result might not come as a surprise (though, surprising the reader is not any of the aims of the paper), as the definition of the congruence extension property is second-order. On the other hand, we know that certain first-order properties (e.g. additivity) imply the CEP, and it is not that clear that certain others do not imply the negation of the CEP. Thus there could be, in principle, some interesting interplay between the CEP and sets of first-order formulas. We settle this question.}

\section{A warm-up example}

An algebra $\gA$ has the \emph{congruence extension property} (CEP) if for any 
subalgebra $\gB\subseteq\gA$ and any congruence $\Theta\in\Con(\gB)$ there
is $\Psi\in\Con(\gA)$ such that $\Theta = \Psi\cap(B\times B)$. 
A class of algebras has the CEP if all its members have the CEP. \\

As already mentioned, if $(A, f)$ is a Boolean frame with an additive $f$, then $(A, f)$ has the CEP.  
But it is not only additivity that implies the CEP. Here we give an example for a strictly weaker identity
which does imply the CEP. By strictly weaker we mean not only that there are non-additive functions satisfying our
identity, but also that the identity can be satisfied by functions that are not term-equivalent to any additive operation. \\

Recall e.g. from \cite[Sec. 5]{Krawczyk2023} or \cite[p.223]{chagrov1997modal} that if 
$\Theta\in\Con(\gA)$ for a Boolean frame $\gA$, then 
\[ I = \{ x\oplus y:\; (x,y)\in\Theta\}\subseteq A\]
is a congruential ideal: if $x\oplus y\in I$, then $f(x)\oplus f(y)\in I$. ($\oplus$ denotes the symmetric difference).
Conversely, for any congruential ideal
$I$ the relation $\Theta = \{ (x,y):\; x\oplus y\in I\}$ is a congruence of $\gA$. 

The proof that additivity implies the CEP is based on this argument: Take an additive $\gA$, a subalgebra $\gB\subseteq \gA$
and a congruence $\Theta_{\gB}$ of $\gB$. Then $\Theta_{\gB}$ gives rise to a congruential ideal $I_{\gB}$. 
Additivity guarantees that the \emph{ideal extension} $I$ of $I_{\gB}$ to an ideal in $\gA$ is congruential. Inspecting this
proof shows that instead of additivity the following formula is enough to assume:

\begin{fact}\label{fact:1}
	Any Boolean frame $\gA$ that satisfies 
    \begin{align}
        x\oplus y \leq z \implies f(x)\oplus f(y) \leq f(z) 
        \label{gyengites:old}
    \end{align}
    has the CEP. 
\end{fact}

\noindent Now, if $u\leq v$, then $u\oplus v\leq v$, and 
\labelcref{gyengites:old} gives $f(u)\oplus f(v)\leq f(v)$.
Hence $f(u)\leq f(v)$. Now $(u\OR v)\oplus u\leq v$, so a second application of \labelcref{gyengites:old} gives
$f(u\OR v)\oplus f(u)\leq f(v)$. Thus $f(u\OR v)\leq f(u)\OR f(v)$. Consequently, \labelcref{gyengites:old} is in fact
equivalent to additivity. 

To get a weaker identity that still implies CEP
the idea is to modify the bound $f(z)$ on the right-hand side
of \labelcref{gyengites:old}. More precisely, we consider
\begin{align}
    x\oplus y \leq z \implies 
    f(x)\oplus f(y) \leq \bigl(f(z)\oplus f(0)\bigr)
        \OR
        \bigl(f(-z)\oplus f(1)\bigr).    
    \tag{$*$}
    \label{gyengites}
\end{align}
We claim that \labelcref{gyengites} is equivalent to an identity that is strictly weaker than additivity, and that it implies the CEP.
First, define
\[
    d(z)=\bigl(f(z)\oplus f(0)\bigr)\OR
    \bigl(f(-z)\oplus f(1)\bigr),
\]
Then the identity 
\[
    f(x)\oplus f(y)\leq d\bigl(z\OR(x\oplus y)\bigr)
\]
is equivalent to \labelcref{gyengites}. To see that additivity implies \labelcref{gyengites}, suppose $f$ is additive, and consider $h(x)=f(x)\AND -f(0)$. Then $h$ is normal and additive, and $f(x)=f(0)\OR h(x)$. If $x\oplus y\leq z$, then additivity gives $h(x)\leq h(y)\OR h(z)$, and likewise with $x$ and $y$ interchanged. Thus, $f(x)\oplus f(y)=h(x)\oplus h(y)\leq h(z)=f(z)\oplus f(0)$.

\begin{fact}\label{fact:2}
    Any Boolean frame $\gA$ that satisfies 
    \labelcref{gyengites} has the CEP.
\end{fact}
\begin{proof}
    Let $B\leq A$, let $\theta\in\Con(B,f)$, and put $I=[0]_\theta$. Observe that $d(0)=0$. Thus, for every $b\in I$, compatibility with the term $d$ gives $d(b)\equiv_\theta d(0)=0$, so $d(b)\in I$.

    Let $J$ be the Boolean ideal of $A$ generated by $I$. If $x\oplus y\in J$, then $x\oplus y\leq b$ for some $b\in I$. Then \labelcref{gyengites} yields
    \[
        f(x)\oplus f(y)\leq d(b)\in I.
    \]
    Hence $f(x)\oplus f(y)\in J$, so the Boolean congruence determined by $J$ is compatible with $f$. Since $J\cap B=I$, it extends $\theta$.
\end{proof}

We give an example for an algebra $(\gA,f)$ that satisfies \labelcref{gyengites} but $f$ is not just non-additive, 
but it is not term-equivalent to any additive $g$. 

\begin{example}\label{exmp1}\label{thm:warmingup}
    Consider the powerset Boolean algebra $\gC = \mathfrak{P}(\{1,2,3\})$ and define
    \[
        \Diamond x= \begin{cases}0&x=0,\\1&x\ne0,\end{cases}
                \qquad
        \Box x=\begin{cases}1&x=1,\\0&x\ne1.\end{cases}
    \]
    Let $\gA=(A,f)$ be the product 
    \[
        \gA=(C,\Diamond)\times(C,\Box).
    \]
    Then $\gA$ is normal, monotone, and satisfies \labelcref{gyengites}. Therefore
    $\gA$ has the CEP. On the other hand, $\gA$ 
    is not term-equivalent to any additive $(A, g)$.
\end{example}
\begin{proof}
    That $\gA$ is normal, monotone and satisfies
    \labelcref{gyengites} is routine to check and we skip
    the details. By Fact \ref{fact:2}, $\gA$ has the CEP. 

    We show that $f$ is not term-equivalent to any
    additive $g$. In each factor of $\gA$ the only additive unary term operations are $0,1,x$, and $\Diamond x$. Indeed, on the nonzero proper elements, every unary term agrees uniformly with one of $0,1,x,-x$, by induction on terms. An additive non-constant term takes $0$ to $0$ and $1$ to $1$. Its value on the non-zero proper elements cannot be constantly $0$, since two complementary elements join to $1$. It cannot agree with $-x$ either, since $C$ has comparable non-zero proper elements and the term is monotone. That leaves $x$ and $\Diamond x$.

    Suppose an additive operation $g$ was term-equivalent to $f$. In each factor, $g$ must define the original operation $\Diamond$ or $\Box$ (depending on the factor). The choices $0,1,x$ yield only Boolean terms, so $g$ must act as $\Diamond$ in both factors. But then every term in $g$ acts identically in the two factors, whereas the original operation acts as $\Diamond$ in the first and $\Box$ in the second. This is the contradiction.
\end{proof}

By taking the variety $\mathbf{V}(\gA)$ we get:

\begin{corollary}
	There is a strongly non-additive variety of Boolean frames satisfying the CEP. 
\end{corollary}


The identity \labelcref{gyengites} provides a useful condition below additivity. While additivity implies  CEP, there exist continuum many monotone varieties without the CEP (cf. \cite{GyenisMolnarNDJFL}). Since 
\labelcref{gyengites} implies  CEP but is not, in general, interdefinable with additive functions, additivity  is too strong, whereas monotonicity alone is too weak for varieties to have the CEP.\\

\noindent With the following construction in modal logic, it is easy to obtain continuum many strongly non-additive, but monotone varieties with the CEP. Recall the definition of a wheel frame from \cite{Miyazaki2005}. For $n\geq 5$ the wheel frame $\mathcal{W}_n = \<W_n, R_n\>$ is the frame, where
\begin{align*}
    W_n&\defeq\{ 0,\ldots, n-1 \}\cup\{h\} \\
    R_n&\defeq\{ \<x,y\>: x, y<n, |x-y|\leq 1\ (\text{mod } n)\}\cup\{ \<h,h\>, \<h,x\>, \<x,h\>: x<n\}.
\end{align*}
    \begin{center}
        \begin{tikzpicture}[thick,scale=.2]
            \graph [nodes={draw,circle,inner sep=.2pt,minimum size=5pt, fill}, clockwise, radius=.2cm, empty nodes]
  {subgraph A[at={(0,-1)}] -- subgraph C_n[n=9]};
        \end{tikzpicture}

        The frame $\mathcal{W}_9$
    \end{center}

\begin{theorem}
    There are continuum many strongly non-additive, monotone varieties with the CEP.
\end{theorem}
\begin{proof}
 Let $Prim = \{n\in \omega: \text{ $n$ is prime $n\geq 7$}\}$ and $\mathfrak{W}_n=\mathcal{W}_n^+$ be the complex algebra. For $X\subseteq Prim$ we write \[\mathbf{V}(X) =\mathbf{V}\big(\{\mathfrak{W}_p: p\in X\}\cup\{\gA\}\big),\] where $\gA$ is from Example \labelcref{exmp1}. We show that for all $X\neq Y\subsetneq Prim$, where  $X,Y$ are non-empty, we have \[\mathbf{V}(X)\neq \mathbf{V}(Y).\]
 Monotonicity and CEP follow from the fact that the generators satisfy the identity \labelcref{gyengites}, while strong-non additivity follows from Theorem \labelcref{thm:warmingup}. Now let us show that the varieties are different. Since each $\mathcal{W}_n$ is point-generated,  $\mathfrak{W}_n$ is subdirectly irreducible. 
 Thus, by Jónsson's Lemma, it is enough to see that for $p\in X\setminus Y$ we have $\mathfrak{W}_p\not\in \mathbf{HSP}_U\big(\{\mathfrak{W}_q: q\in Y\}\cup\{\gA\}\big)$. Otherwise, if $\mathfrak{W}_p\in \mathbf{V}(Y)_{SI} $ for some $p\in X$, then there would be  some $\gD\in \mathbf{P}_U(\{\mathfrak{W}_q: q\in Y\}\cup\{\gA\})$ such that $\mathfrak{W}_p\in \mathbf{HS}(\gD)$. Moreover, either  $\gD\cong \gA$ or $\gD$ would be an ultraproduct of some $\mathfrak{W}_q$'s from $\mathbf{V}(Y)$. The latter  cannot occur, since, by \cite[Theorem 21]{Miyazaki2005} for any $p\in Prim$ there is a formula $A_p$, such that $\mathcal{W}_p\not\Vdash A_p$, while $\mathcal{W}_q\Vdash A_p$, for all $p\neq q\in Prim$. The former case  is also impossible, as $|\mathfrak{W}_p| > |\gA|$, by construction.
\end{proof}

\section{Strongly non-additive logics with deduction}

Let us start by recalling (e.g. from \cite{BurrisSankappanavar1981}) some of the important definitions 
and theorems from universal algebra that we make use of in this section. 
\medskip

A variety $\mathbf{V}$ has \emph{equationally definable principal congruences} (EDPC) if there exist 
finitely many pairs of quaternary terms
\[
	p_1,q_1,\dots,p_n,q_n
\]
such that for every algebra $\gA\in\mathbf{V}$ and all $a,b,c,d\in A$,
\[
	(c,d)\in \operatorname{Cg}^{\mathfrak A}(a,b)
		\quad\Longleftrightarrow\quad
		\mathfrak A\models \bigwedge_{i=1}^n p_i(a,b,c,d) = q_i(a,b,c,d).
\]
Equivalently, the principal congruence generated by a pair $(a,b)$ is uniformly definable throughout $\mathbf{V}$ 
by a conjunction of equations.

A ternary term $t(x,y,z)$ is a \emph{discriminator term} on $\mathfrak A$ if
\[
	t^{\mathfrak A}(a,b,c)=
	\begin{cases}
		a & \text{ if } a\neq b,\\[2mm]
		c & \text{ if } a=b,
	\end{cases}
	\qquad(a,b,c\in A).
\]
A variety $\mathbf V$ is a \emph{discriminator variety} if there exists a ternary term $t(x,y,z)$ such that, 
for every subdirectly irreducible algebra $\mathfrak A\in\mathbf V$, the interpretation $t^{\mathfrak A}$ is a 
discriminator operation on $\mathfrak A$. In classes of algebras which have a Boolean algebra reduct, the 
discriminator term can be replaced with the so called switching term
\[
	d^{\gA}(x) = \begin{cases}
		1 & \text{ if } x\neq 0,\\
		0 & \text{ if } x=0.
	\end{cases}
\]
By this we mean that any switching term $d^{\gA}$ gives rise to a discriminator term $t^{\gA}$ and vice-versa (see \cite{AGyNS2022}).
Every discriminator variety has EDPC and CEP \cite{BurrisSankappanavar1981}.

\begin{theorem}\label{thm:m1}
	There are continuum many strongly non-additive minimal discriminator varieties. 
\end{theorem}
\begin{proof}
	Let $B$ be the Boolean algebra of finite-cofinite subsets of $\mathbb{N}$
	\[
		B = \{ X\subseteq\mathbb{N}:\; X\text{ is finite or cofinite}\},
	\]
	with atoms $a_n = \{n\}$, co-atoms $b_n = \mathbb{N}-a_n$,  
	and $u_n = \{0, \ldots, n\}$ for $n\in\mathbb{N}$. 
	For $S\subseteq\mathbb{N}\setminus \{0\}$ define the unary operation $f_S:B\to B$ by
	\[
		f_S(x) = \begin{cases}
				b_0        & \text{ if } x=\mathbb{N},\\
				b_{n+1}    & \text{ if } x=b_n,\\
				b_n        & \text{ if } x=a_n,\\
				-u_n       & \text{ if } x=u_n\text{ and } n\in S\ (n\ge 1),\\
				\mathbb{N} & \text{ if } x=u_n\text{ and } n\notin S\ (n\ge 1),\\
				\emptyset  & \text{ if } x=\emptyset,\\
				\mathbb{N} & \text{ otherwise}.
		\end{cases}
	\]
	Let $\gA_S = (B, f_S)$ be the corresponding Boolean frame, and $\mathbf{V}_S = \mathbf{V}(\gA_S)$ the
	variety generated by $\gA_S$. \\

	\noindent {\bf Claim:} Each $\mathbf{V}_S$ is a discriminator variety. This is because
	each $\gA_S$ has a switching term
	\[
		d(x) = x\OR f(x)\,.
	\]
	Indeed, $d(\emptyset)=\emptyset$ and for any $x\neq\emptyset$, $d(x)=\mathbb{N}$. Since $\gA_S$ has a Boolean reduct, 
	$d$ gives rise to a discriminator term (see \cite[2.7]{AGyNS2022}). It follows that each $\mathbf{V}_S$ is a discriminator variety. \\

	\noindent {\bf Claim:} Each $\mathbf{V}_S$ is minimal. Note first that the entire algebra $\gA_S$ is $\emptyset$-generated:
	\[
		b_n = f_S^{(n+1)}(\mathbb{N}), \qquad a_n = -b_n \text{ for } n\in\mathbb{N},
	\]
	and thus the Boolean operations generate every finite or cofinite subset of $\mathbb{N}$. 
	Next, as $\gA_S$ has a discriminator term, $\gA_S$ is simple. Consider the free 
	algebra $\mathfrak{F}_S(\emptyset)$ generated by the empty set in the variety $\mathbf{V}_S$ 
	(it exists, because the similarity type contains nullary symbols). By the universal 
	property of the free algebra there is a canonical homomorphism
	\[
		h: \mathfrak{F}_S(\emptyset) \to \gA_S\,.
	\]
	The image is the subalgebra of $\gA_S$ generated by the nullary symbols, which is exactly the entire $\gA_S$, hence $h$ is surjective. 
	Because $\gA_S$ generates the variety, if for two terms $t^{\gA}=s^{\gA}$, then this holds in the free algebra too: $t^{\mathfrak{F}_S(\emptyset)}=s^{\mathfrak{F}_S(\emptyset)}$, meaning that $h$ must be injective. Therefore, $h$
	is an isomorphism $\mathfrak{F}_S(\emptyset)\cong  \gA_S$. Let now $\mathbf{W}\subseteq\mathbf{V}_S$ be a subvariety. Then
	there is a canonical surjection 
	\[
		\mathfrak{F}_S(\emptyset) \twoheadrightarrow \mathfrak{F}_{\mathbf{W}}(\emptyset)\,.
	\]
	But as $\mathfrak{F}_S(\emptyset)\cong  \gA_S$, $\mathfrak{F}_S(\emptyset)$ is simple, and its homomorphic images are either
	trivial or itself. If $\mathbf{W}$ is non-trivial, then $\mathfrak{F}_{\mathbf{W}}(\emptyset)$ is nontrivial, hence
	\[
		\mathfrak{F}_{\mathbf{W}}(\emptyset)\cong \mathfrak{F}_S(\emptyset)\cong  \gA_S\,.
	\]
	In this case $\gA_S\in\mathbf{W}$ and therefore $\mathbf{W} = \mathbf{V}_S$, meaning that $\mathbf{V}_S$ has no proper nontrivial
	subvarieties.\\

	\noindent {\bf Claim:} The operation $f_S$ of $\gA_S$ cannot be interdefined with any additive operation $g$. 
	Indeed, if $g$ is additive, and for some term $\tau\in \operatorname{Term}($Boolean, $g)$ we have
	$\tau(x) = f_S(x)$, then by Makinson \cite[Theorem 2]{makinson1971} we would have $\mathbf{2}\in \mathbf{V}(\gA_S)$. By minimality
	of $\mathbf{V}_S$ this would imply $\mathbf{V}(\mathbf{2}) = \mathbf{V}_S$, contradicting the existence of an infinite 
	$\emptyset$-generated algebra $\gA_S\in\mathbf{V}_S$. \\

	\noindent {\bf Claim:} We have continuum many distinct $\mathbf{V}_S$. For each $n\ge 1$ we have
	\[
		u_n = \bigsqcup_{i=0}^{n}-f^{(i+1)}(\mathbb{N}),
	\]
	because $f^{(i+1)}(\mathbb{N}) = b_i$ and $-b_i = a_i$. Now the equation
	\[
		f(u_n)\AND u_n = 0
	\]
	holds in $\gA_S$ iff $f_S(u_n) = -u_n$ iff $n\in S$. Hence, $S\mapsto \mathbf{V}_S$ is injective, 
	giving $2^{\aleph_0}$ pairwise distinct minimal discriminator varieties. 
\end{proof}

We actually proved more: Makinson's result applies not only to additive but also to monotone operations. Therefore, 
there are continuum many strongly non-monotone minimal discriminator varieties. Here strongly non-monotone means
that the variety is not term-equivalent to any variety where the operation is monotonic. \\

The translation of Theorem \labelcref{thm:m1} to logic is as follows.

\begin{corollary}
	There are continuum many strongly non-additive (in fact, non-monotone) maximal modal logics 
	having a (global) deduction-detachment theorem.
\end{corollary}

Next we show that strongly non-additive modal logics with the (global) deduction-detachment theorem abound. Let $\NVar$ be the lattice of normal subvarieties of Boolean frames, $\TVar$ the lattice of subvarieties of Boolean frames where $f(1)=1$, and $\Sna$ the lattice of strongly non-additive subvarieties of Boolean frames.

\begin{theorem}\label{functor}
	There is a lattice embedding functor $\mathbf{T}:\NVar\to\Sna$ such that for every normal variety $\mathbf{W}$ of Boolean frames
	$\mathbf{T}(\mathbf{W})$ is a strongly non-additive variety, and
	\begin{itemize}
		\item $\mathbf{T}$ is injective, and
		\item $\mathbf{T}$ preserves CEP, EDPC and the property of being discriminator variety.
	\end{itemize}
\end{theorem}
\begin{proof}
	Using the convention $n=\{0, 1, \ldots, n-1\}$ let $\gA$ be the algebra $(\mathfrak{P}(n), f)$ 
	for some finite $n\geq 3$ and $f$ defined as
	\[
	    \emptyset \mapsto n \mapsto \{n-1\} \mapsto \{n-2\} \mapsto \cdots \mapsto \{0\} \mapsto \emptyset
	\]
	and $f(x)=x$ for all other elements. \\
	
	\noindent {\bf Claim:} $\mathbf{HS}(\gA) = \{\mathbf{1},\gA\}$. This is immediate as $\gA$ has no proper subalgebras, 
	and is simple. \\
	
	\noindent {\bf Claim:} $\mathbf{V}(\gA)$ is a discriminator variety, it has EDPC and CEP. As $\gA$ is simple and 
	has no proper subalgebras, it is hereditarily simple (every subalgebra of $\gA$ is simple). Every variety of 
	Boolean frames is arithmetical, and thus by Pixley \cite{BurrisSankappanavar1981}
	$\gA$ is quasi-primal, and thus it has a discriminator term. Therefore, $\mathbf{V}(\gA)$ is a discriminator variety.
	Every discriminator variety has EDPC and CEP, see \cite{BurrisSankappanavar1981}. \\
	
	\noindent {\bf Claim:} $\gA$ is strongly non-additive. Indeed, if $g$ was additive, and for some term 
	$\tau\in \operatorname{Term}($Boolean, $g)$ we had $\tau(x) = f(x)$, then by Makinson \cite[Theorem 2]{makinson1971} 
	we would have $\mathbf{2}\in \mathbf{HS}(\gA)$ contradicting $\mathbf{HS}(\gA) = \{\mathbf{1},\gA\}$.\\
	
	\noindent Take now a variety $\mathbf{W}$ in which $f(0)=0$ holds and consider the join of the varieties
	\[
	 	\mathbf{T}(\mathbf{W}) = \mathbf{W}\lor \mathbf{V}(\gA),
	\]
    and the binary term 
    \[
        s(x,y) = (x\AND -f(0))\OR (y\AND f(0))\,.
    \]
	Then 
	\begin{align*}
		\mathbf{W}      &\models s(x,y) = (x\AND -0)\OR (y\AND 0) = x, \\
		\mathbf{V}(\gA) &\models s(x,y) = (x\AND -1)\OR (y\AND 1) = y.
	\end{align*}
	This means that the two varieties are independent (see \cite{gratzerlakserplonka}) and 
	\[
		\mathbf{W}\lor \mathbf{V}(\gA) = \mathbf{W}\times \mathbf{V}(\gA).
	\] 
	Also,  
    \[
        \mathbf{W} = \mathbf{T}(\mathbf{W}) \land \{f(0)=0\},
        \qquad
        \mathbf{V}(\gA) = \mathbf{T}(\mathbf{W}) \land \{f(0)=1\}.
    \]
    and therefore the operation $\mathbf{T}$ is injective. That 
    \[
        \mathbf{T}(\mathbf{W_1}\lor \mathbf{W_2})
        =
        \mathbf{T}(\mathbf{W_1})\lor \mathbf{T}(\mathbf{W_2}),\qquad
        \mathbf{T}(\mathbf{W_1}\land \mathbf{W_2})
        =
        \mathbf{T}(\mathbf{W_1})\land \mathbf{T}(\mathbf{W_2})
    \]
    is straightforward from the definition and independence of $\mathbf{V}(\gA)$ from any normal subvariety. 
	
	It remains to show that $\mathbf{T}$ preserves EDPC, CEP and the property of being a discriminator variety.
	For CEP this is theorem 3.2(b) \cite{compendium}. If $t_1$ and $t_2$ are discriminator terms respectively in
	$\mathbf{W}$ and $\mathbf{V}(\gA)$, then $s(t_1, t_2)$ is a discriminator term in $\mathbf{T}(\mathbf{W})$.
	A similar argument applies for the preservation of EDPC once it is noted that congruence-distributive 
	varieties (in particular, varieties of Boolean frames) have the Fraser--Horn property: congruences of products are rectangular,
	that is congruences of a finite product of algebras are the products of the congruences \cite{fraserhorn}.\\
	
	Finally, if the assumption of normality $f(0)=0$ is replaced by $f(1)=1$, then we only need to alter the
	definition of $\gA$, and define $f$ by ``reversing the arrows'':
	\[
	    \emptyset \mapsto \{0\} \mapsto \{1\} \mapsto \cdots
	    \mapsto \{n-1\}\mapsto n\mapsto \emptyset
	\]
	and $f(x)=x$ for all other elements.
\end{proof}

\begin{corollary}
    Every normal modal logic admits a strongly non-additive axiomatic reduction; this reduction preserves the local deduction-detachment theorem, and it preserves the global deduction-detachment theorem whenever the original logic has it.
    For pairwise different normal modal logics the reductions are pairwise different.
\end{corollary}

\section{SHhsSH}

If an algebra $\gA$ has the CEP a well-known consequence is that $\mathbf{HS}(\gA)=\mathbf{SH}(\gA)$. We say that a class of algebras $\mathbf{K}$ has $\mathbf{HS}=\mathbf{SH}$, if each $\gA\in \mathbf{K}$ has $\mathbf{HS}(\gA)=\mathbf{SH}(\gA)$. In this section we briefly investigate the problem of how much a Boolean frame or variety \textit{lacking} CEP does \textit{not} ruin  $\mathbf{HS}=\mathbf{SH}$. We begin with a general transfer result concerning finite products in congruence-distributive varieties.

\begin{lemma}\label{hs=sh}
    Let $\gA$ and $\gB$ be algebras in a congruence-distributive variety. If
    \[
    \mathbf{HS}(\gA)=\mathbf{SH}(\gA)
    \qquad\text{and}\qquad
    \mathbf{HS}(\gB)=\mathbf{SH}(\gB),
    \]
    then
    \[
    \mathbf{HS}(\gA\times \gB)=\mathbf{SH}(\gA\times \gB).
    \]
    In particular, this holds for Boolean frames.
\end{lemma}
\begin{proof}
    For every algebra $\gX$ one always has $\mathbf{SH}(\gX)\subseteq \mathbf{HS}(\gX)$, thus it suffices to prove that
    $\mathbf{HS}(\gA\times \gB)\subseteq \mathbf{SH}(\gA\times \gB)$.
    Let $\gD\in \mathbf{HS}(\gA\times \gB)$. Then there exist a subalgebra $\gC\subseteq \gA\times \gB$ and a congruence $\theta\in\Con(\gC)$ such that $\gD\cong \gC/\theta$.
    Let $\pi_A\colon A\times B\to A$ and $\pi_B\colon A\times B\to B$ be the coordinate projections, and define
    \[
    \alpha=\ker(\pi_A\upharpoonright C),
    \qquad
    \beta=\ker(\pi_B\upharpoonright C).
    \]
    Since $C\subseteq A\times B$, we have $\alpha\wedge \beta = \Delta_{\gC}$.
    Consider the homomorphism
    \[
    i: \gC/\theta \longrightarrow
    \gC/(\theta\vee\alpha)\times \gC/(\theta\vee\beta),
    \qquad
    c/\theta \longmapsto
    \bigl(c/(\theta\vee\alpha),\,c/(\theta\vee\beta)\bigr).
    \]
    Its kernel is, by the second isomorphism theorem,
    \[
    \ker(i)=
    \nicefrac{(\theta\vee\alpha)\wedge(\theta\vee\beta)}{\theta}.
    \]
    By congruence distributivity,
    \[
    (\theta\vee\alpha)\wedge(\theta\vee\beta)
    =
    \theta\vee(\alpha\wedge\beta)
    =
    \theta\vee \Delta_{\gC}
    =
    \theta.
    \]
    Hence $\ker(i)=\Delta_{\gC/\theta}$, and therefore $i$ is injective. It follows that
    \[
    \gC/\theta \in
    \mathbf{S}\bigl(\gC/(\theta\vee\alpha)\times \gC/(\theta\vee\beta)\bigr).
    \]
    Since $\theta\leq \theta\vee\alpha$, the quotient
    $\gC/(\theta\vee\alpha)$ is a homomorphic image of $\gC/\alpha$, and so $\gC/(\theta\vee\alpha)\in \mathbf{H}(C/\alpha)$.
    But $\gC/\alpha\cong \pi_A(C)$, and $\pi_A(\gC)\subseteq \gA$. Thus
    \[
    \gC/(\theta\vee\alpha)\in \mathbf{HS}(\gA)=\mathbf{SH}(\gA).
    \]
    Therefore, there exist $\varphi_{\gA}\in\Con(\gA)$ and a subalgebra $\gA_0\subseteq \gA/\varphi_{\gA}$ such that
    $\gC/(\theta\vee\alpha)\cong \gA_0$.
    Similarly, \[\gC/(\theta\vee\beta)\in \mathbf{HS}(\gB)=\mathbf{SH}(\gB),\] so there exist $\varphi_{\gB}\in\Con(\gB)$ and a subalgebra
    $\gB_0\subseteq \gB/\varphi_{\gB}$ such that $\gC/(\theta\vee\beta)\cong \gB_0$.
    We obtained $\gC/\theta \in \mathbf{S}(\gA_0\times \gB_0)$.
    As $\gA_0\subseteq \gA/\varphi_{\gA}$ and $\gB_0\subseteq \gB/\varphi_{\gB}$, it follows that
    \[
    \gA_0\times \gB_0 \subseteq (\gA/\varphi_{\gA})\times (\gB/\varphi_{\gB}).
    \]
    Hence
    \[
    \gC/\theta \in
    \mathbf{S}\bigl((\gA/\varphi_\gA)\times(\gB/\varphi_\gB)\bigr).
    \]
    By the Fraser-Horn property of congruence distributive varieties we have
    \[
    (\gA/\varphi_\gA)\times(\gB/\varphi_\gB)\cong
    (\gA\times \gB)/(\varphi_\gA\times\varphi_\gB),
    \]
    showing that $(\gA/\varphi_\gA)\times(\gB/\varphi_\gB)\in \mathbf{H}(\gA\times \gB)$, and therefore $\gC/\theta\in \mathbf{SH}(\gA\times \gB)$.
\end{proof}

\noindent  We will make use of the following example.
\begin{example}\label{exmp}
    
 Let $A$ be the 8-element Boolean algebra, and consider the Boolean frame $\gA =(A,f)$ depicted below (fixpoints of $f$ are omitted):

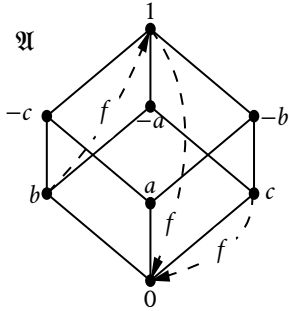
\begin{center}

\tikzset{every picture/.style={line width=0.75pt}} 

\begin{tikzpicture}[x=0.75pt,y=0.75pt,yscale=-1,xscale=.8]

\draw   (286.15,57.51) -- (350.9,101.53) -- (286.15,145.55) -- (221.39,101.53) -- cycle ;
\draw   (286.15,18.38) -- (350.9,62.4) -- (286.15,106.42) -- (221.39,62.4) -- cycle ;
\draw    (221.39,62.4) -- (221.39,101.53) ;
\draw    (350.9,62.4) -- (350.9,101.53) ;
\draw    (286.15,106.42) -- (286.15,145.55) ;
\draw    (286.15,18.38) -- (286.15,57.51) ;
\draw [line width=0.75]  [dash pattern={on 4.5pt off 4.5pt}]  (221.39,101.53) .. controls (256.18,72.13) and (267.09,56.26) .. (285.31,20.05) ;
\draw [shift={(286.15,18.38)}, rotate = 116.6] [fill={rgb, 255:red, 0; green, 0; blue, 0 }  ][line width=0.08]  [draw opacity=0] (12,-3) -- (0,0) -- (12,3) -- cycle    ;
\draw  [dash pattern={on 4.5pt off 4.5pt}]  (286.15,18.38) .. controls (329.23,63.73) and (296.81,123.5) .. (286.99,143.78) ;
\draw [shift={(286.15,145.55)}, rotate = 295.08] [fill={rgb, 255:red, 0; green, 0; blue, 0 }  ][line width=0.08]  [draw opacity=0] (12,-3) -- (0,0) -- (12,3) -- cycle    ;
\draw  [dash pattern={on 4.5pt off 4.5pt}]  (350.9,101.53) .. controls (346.64,102.53) and (362.51,129.38) .. (287.29,145.31) ;
\draw [shift={(286.15,145.55)}, rotate = 348.27] [fill={rgb, 255:red, 0; green, 0; blue, 0 }  ][line width=0.08]  [draw opacity=0] (12,-3) -- (0,0) -- (12,3) -- cycle    ;
\draw    (350.9,101.53) ;
\draw [shift={(350.9,101.53)}, rotate = 0] [color={rgb, 255:red, 0; green, 0; blue, 0 }  ][fill={rgb, 255:red, 0; green, 0; blue, 0 }  ][line width=0.75]      (0, 0) circle [x radius= 2.68, y radius= 2.68]   ;
\draw    (286.15,57.51) ;
\draw [shift={(286.15,57.51)}, rotate = 0] [color={rgb, 255:red, 0; green, 0; blue, 0 }  ][fill={rgb, 255:red, 0; green, 0; blue, 0 }  ][line width=0.75]      (0, 0) circle [x radius= 2.68, y radius= 2.68]   ;
\draw    (286.15,18.38) ;
\draw [shift={(286.15,18.38)}, rotate = 0] [color={rgb, 255:red, 0; green, 0; blue, 0 }  ][fill={rgb, 255:red, 0; green, 0; blue, 0 }  ][line width=0.75]      (0, 0) circle [x radius= 2.68, y radius= 2.68]   ;
\draw    (221.39,101.53) ;
\draw [shift={(221.39,101.53)}, rotate = 0] [color={rgb, 255:red, 0; green, 0; blue, 0 }  ][fill={rgb, 255:red, 0; green, 0; blue, 0 }  ][line width=0.75]      (0, 0) circle [x radius= 2.68, y radius= 2.68]   ;
\draw    (221.39,62.4) ;
\draw [shift={(221.39,62.4)}, rotate = 0] [color={rgb, 255:red, 0; green, 0; blue, 0 }  ][fill={rgb, 255:red, 0; green, 0; blue, 0 }  ][line width=0.75]      (0, 0) circle [x radius= 2.68, y radius= 2.68]   ;
\draw    (286.15,106.42) ;
\draw [shift={(286.15,106.42)}, rotate = 0] [color={rgb, 255:red, 0; green, 0; blue, 0 }  ][fill={rgb, 255:red, 0; green, 0; blue, 0 }  ][line width=0.75]      (0, 0) circle [x radius= 2.68, y radius= 2.68]   ;
\draw    (350.9,62.4) ;
\draw [shift={(350.9,62.4)}, rotate = 0] [color={rgb, 255:red, 0; green, 0; blue, 0 }  ][fill={rgb, 255:red, 0; green, 0; blue, 0 }  ][line width=0.75]      (0, 0) circle [x radius= 2.68, y radius= 2.68]   ;
\draw    (286.15,145.55) ;
\draw [shift={(286.15,145.55)}, rotate = 0] [color={rgb, 255:red, 0; green, 0; blue, 0 }  ][fill={rgb, 255:red, 0; green, 0; blue, 0 }  ][line width=0.75]      (0, 0) circle [x radius= 2.68, y radius= 2.68]   ;

\draw (286.15,148.95) node [anchor=north] [inner sep=0.75pt]    {$0$};
\draw (286.15,14.98) node [anchor=south] [inner sep=0.75pt]    {$1$};
\draw (219.39,101.53) node [anchor=east] [inner sep=0.75pt]    {$b$};
\draw (352.9,62.4) node [anchor=west] [inner sep=0.75pt]    {$-b$};
\draw (286.15,103.02) node [anchor=south] [inner sep=0.75pt]    {$a$};
\draw (286.15,60.91) node [anchor=north] [inner sep=0.75pt]    {$-a$};
\draw (355.82,101.53) node [anchor=west] [inner sep=0.75pt]    {$c$};
\draw (214.44,61.62) node [anchor=east] [inner sep=0.75pt]    {$-c$};
\draw (208.13,24.62) node    {$\mathfrak{A}$};
\draw  [color={rgb, 255:red, 255; green, 255; blue, 255 }  ,draw opacity=1 ][fill={rgb, 255:red, 255; green, 255; blue, 255 }  ,fill opacity=1 ]  (257.45, 60.18) circle [x radius= 9.9, y radius= 14.14]   ;
\draw (263.45,67.78) node [anchor=south east] [inner sep=0.75pt]    {$f$};
\draw  [color={rgb, 255:red, 255; green, 255; blue, 255 }  ,draw opacity=1 ][fill={rgb, 255:red, 255; green, 255; blue, 255 }  ,fill opacity=1 ]  (331.03, 131.08) circle [x radius= 9.19, y radius= 13.44]   ;
\draw (336.53,138.18) node [anchor=south east] [inner sep=0.75pt]    {$f$};
\draw  [color={rgb, 255:red, 255; green, 255; blue, 255 }  ,draw opacity=1 ][fill={rgb, 255:red, 255; green, 255; blue, 255 }  ,fill opacity=1 ]  (298.38, 116.85) circle [x radius= 9.9, y radius= 14.14]   ;
\draw (304.38,124.45) node [anchor=south east] [inner sep=0.75pt]    {$f$};

\end{tikzpicture}

\end{center}

\noindent It is easy to see that the subalgebras of $\gA$ are among

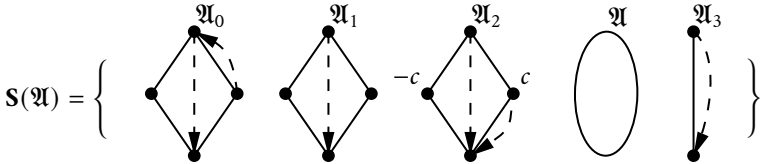
\begin{center}

\tikzset{every picture/.style={line width=0.75pt}} 

\begin{tikzpicture}[x=0.75pt,y=0.75pt,yscale=-1,xscale=1]

\draw   (181.32,208.19) -- (202.78,239.53) -- (181.32,270.87) -- (159.87,239.53) -- cycle ;
\draw   (248.41,208.19) -- (269.87,239.53) -- (248.41,270.87) -- (226.96,239.53) -- cycle ;
\draw   (319.41,208.19) -- (340.86,239.53) -- (319.41,270.87) -- (297.95,239.53) -- cycle ;
\draw    (181.32,208.19) ;
\draw [shift={(181.32,208.19)}, rotate = 0] [color={rgb, 255:red, 0; green, 0; blue, 0 }  ][fill={rgb, 255:red, 0; green, 0; blue, 0 }  ][line width=0.75]      (0, 0) circle [x radius= 2.68, y radius= 2.68]   ;
\draw    (181.32,270.87) ;
\draw [shift={(181.32,270.87)}, rotate = 0] [color={rgb, 255:red, 0; green, 0; blue, 0 }  ][fill={rgb, 255:red, 0; green, 0; blue, 0 }  ][line width=0.75]      (0, 0) circle [x radius= 2.68, y radius= 2.68]   ;
\draw    (202.78,239.53) ;
\draw [shift={(202.78,239.53)}, rotate = 0] [color={rgb, 255:red, 0; green, 0; blue, 0 }  ][fill={rgb, 255:red, 0; green, 0; blue, 0 }  ][line width=0.75]      (0, 0) circle [x radius= 2.68, y radius= 2.68]   ;
\draw    (159.87,239.53) ;
\draw [shift={(159.87,239.53)}, rotate = 0] [color={rgb, 255:red, 0; green, 0; blue, 0 }  ][fill={rgb, 255:red, 0; green, 0; blue, 0 }  ][line width=0.75]      (0, 0) circle [x radius= 2.68, y radius= 2.68]   ;
\draw    (248.41,208.19) ;
\draw [shift={(248.41,208.19)}, rotate = 0] [color={rgb, 255:red, 0; green, 0; blue, 0 }  ][fill={rgb, 255:red, 0; green, 0; blue, 0 }  ][line width=0.75]      (0, 0) circle [x radius= 2.68, y radius= 2.68]   ;
\draw    (269.87,239.53) ;
\draw [shift={(269.87,239.53)}, rotate = 0] [color={rgb, 255:red, 0; green, 0; blue, 0 }  ][fill={rgb, 255:red, 0; green, 0; blue, 0 }  ][line width=0.75]      (0, 0) circle [x radius= 2.68, y radius= 2.68]   ;
\draw    (226.96,239.53) ;
\draw [shift={(226.96,239.53)}, rotate = 0] [color={rgb, 255:red, 0; green, 0; blue, 0 }  ][fill={rgb, 255:red, 0; green, 0; blue, 0 }  ][line width=0.75]      (0, 0) circle [x radius= 2.68, y radius= 2.68]   ;
\draw    (248.41,270.87) ;
\draw [shift={(248.41,270.87)}, rotate = 0] [color={rgb, 255:red, 0; green, 0; blue, 0 }  ][fill={rgb, 255:red, 0; green, 0; blue, 0 }  ][line width=0.75]      (0, 0) circle [x radius= 2.68, y radius= 2.68]   ;
\draw    (319.41,208.19) ;
\draw [shift={(319.41,208.19)}, rotate = 0] [color={rgb, 255:red, 0; green, 0; blue, 0 }  ][fill={rgb, 255:red, 0; green, 0; blue, 0 }  ][line width=0.75]      (0, 0) circle [x radius= 2.68, y radius= 2.68]   ;
\draw    (340.86,239.53) ;
\draw [shift={(340.86,239.53)}, rotate = 0] [color={rgb, 255:red, 0; green, 0; blue, 0 }  ][fill={rgb, 255:red, 0; green, 0; blue, 0 }  ][line width=0.75]      (0, 0) circle [x radius= 2.68, y radius= 2.68]   ;
\draw    (297.95,239.53) ;
\draw [shift={(297.95,239.53)}, rotate = 0] [color={rgb, 255:red, 0; green, 0; blue, 0 }  ][fill={rgb, 255:red, 0; green, 0; blue, 0 }  ][line width=0.75]      (0, 0) circle [x radius= 2.68, y radius= 2.68]   ;
\draw    (319.41,270.87) ;
\draw [shift={(319.41,270.87)}, rotate = 0] [color={rgb, 255:red, 0; green, 0; blue, 0 }  ][fill={rgb, 255:red, 0; green, 0; blue, 0 }  ][line width=0.75]      (0, 0) circle [x radius= 2.68, y radius= 2.68]   ;
\draw    (430.97,208.19) ;
\draw [shift={(430.97,208.19)}, rotate = 0] [color={rgb, 255:red, 0; green, 0; blue, 0 }  ][fill={rgb, 255:red, 0; green, 0; blue, 0 }  ][line width=0.75]      (0, 0) circle [x radius= 2.68, y radius= 2.68]   ;
\draw    (430.97,270.87) ;
\draw [shift={(430.97,270.87)}, rotate = 0] [color={rgb, 255:red, 0; green, 0; blue, 0 }  ][fill={rgb, 255:red, 0; green, 0; blue, 0 }  ][line width=0.75]      (0, 0) circle [x radius= 2.68, y radius= 2.68]   ;
\draw  [dash pattern={on 4.5pt off 4.5pt}]  (202.78,239.53) .. controls (200.38,215.75) and (193.94,214.56) .. (183.07,209.09) ;
\draw [shift={(181.32,208.19)}, rotate = 27.89] [fill={rgb, 255:red, 0; green, 0; blue, 0 }  ][line width=0.08]  [draw opacity=0] (12,-3) -- (0,0) -- (12,3) -- cycle    ;
\draw  [dash pattern={on 4.5pt off 4.5pt}]  (181.32,208.19) -- (181.32,268.87) ;
\draw [shift={(181.32,270.87)}, rotate = 270] [fill={rgb, 255:red, 0; green, 0; blue, 0 }  ][line width=0.08]  [draw opacity=0] (12,-3) -- (0,0) -- (12,3) -- cycle    ;
\draw  [dash pattern={on 4.5pt off 4.5pt}]  (248.41,208.19) -- (248.41,268.87) ;
\draw [shift={(248.41,270.87)}, rotate = 270] [fill={rgb, 255:red, 0; green, 0; blue, 0 }  ][line width=0.08]  [draw opacity=0] (12,-3) -- (0,0) -- (12,3) -- cycle    ;
\draw  [dash pattern={on 4.5pt off 4.5pt}]  (321.39,269.58) .. controls (342.02,256.16) and (338.33,258.42) .. (340.86,239.53) ;
\draw [shift={(319.41,270.87)}, rotate = 326.94] [fill={rgb, 255:red, 0; green, 0; blue, 0 }  ][line width=0.08]  [draw opacity=0] (12,-3) -- (0,0) -- (12,3) -- cycle    ;
\draw  [dash pattern={on 4.5pt off 4.5pt}]  (319.41,208.19) -- (319.41,259.3) -- (319.41,268.87) ;
\draw [shift={(319.41,270.87)}, rotate = 270] [fill={rgb, 255:red, 0; green, 0; blue, 0 }  ][line width=0.08]  [draw opacity=0] (12,-3) -- (0,0) -- (12,3) -- cycle    ;
\draw   (386.93,270.87) .. controls (378.31,270.77) and (371.48,256.66) .. (371.68,239.35) .. controls (371.88,222.04) and (379.02,208.09) .. (387.64,208.19) .. controls (396.25,208.29) and (403.08,222.4) .. (402.88,239.71) .. controls (402.69,257.01) and (395.54,270.96) .. (386.93,270.87) -- cycle ;
\draw    (430.97,208.19) -- (430.97,215.09) -- (430.97,270.87) ;
\draw  [dash pattern={on 4.5pt off 4.5pt}]  (430.97,208.19) .. controls (441.67,217.24) and (442.64,242.9) .. (431.66,269.26) ;
\draw [shift={(430.97,270.87)}, rotate = 293.5] [fill={rgb, 255:red, 0; green, 0; blue, 0 }  ][line width=0.08]  [draw opacity=0] (12,-3) -- (0,0) -- (12,3) -- cycle    ;

\draw (85.31,220.61) node [anchor=north west][inner sep=0.75pt]    {$\mathbf{S}(\mathfrak{A}) =\Bigg\{$};
\draw (455.65,220.39) node [anchor=north west][inner sep=0.75pt]    {$\Bigg\}$};
\draw (180.57,207.43) node [anchor=south west] [inner sep=0.75pt]    {$\mathfrak{A}_{0}$};
\draw (247.67,207.43) node [anchor=south west] [inner sep=0.75pt]    {$\mathfrak{A}_{1}$};
\draw (318.66,207.43) node [anchor=south west] [inner sep=0.75pt]    {$\mathfrak{A}_{2}$};
\draw (386.89,207.43) node [anchor=south west] [inner sep=0.75pt]    {$\mathfrak{A}$};
\draw (429.34,207.43) node [anchor=south west] [inner sep=0.75pt]    {$\mathfrak{A}_{3}{}$};
\draw (295.95,236.13) node [anchor=south east] [inner sep=0.75pt]    {$-c$};
\draw (342.86,236.13) node [anchor=south west] [inner sep=0.75pt]    {$c$};

\end{tikzpicture}

\end{center}

\noindent Moreover $\gA$ is simple. The only algebra which is not simple from $\mathbf{S}(\gA)$ is $\gA_2$ with the only non-trivial homomorphic image $\gA_3$. Thus $\mathbf{HS}(\gA) = \mathbf{SH}(\gA)$. However, $\gA$ does not have the CEP: for $\gA_2\subseteq \gA$  the congruence  $\{\{-c,0\}, \{c,1\}\}$ yielding the non-trivial homomorphic image cannot be extended in $\gA$. 

 Using this example, similarly to  Theorem \labelcref{functor} we have the following.

\end{example}

\begin{theorem}
	There is a lattice embedding functor $\mathbf{F}:\mathsf{L}_{f(1)=1}\to\Sna$ such that
	\begin{itemize}
		\item $\mathbf{F}$ is injective, and
		\item $\mathbf{F}$ preserves $\mathbf{HS}=\mathbf{SH}$, but violates CEP.
	\end{itemize}

\end{theorem}
\begin{proof}
     For $\mathbf{W}\in \mathsf{L}_{f(1)=1}$ we let $$\mathbf{F}(\mathbf{W})= \mathbf{W}\vee \mathbf{V}(\gA),$$
    where $\gA$ is from Example \labelcref{exmp}. Similarly to Theorem \labelcref{functor} one can show by using the term
 $$
        s(x,y) = (x\AND f(1))\OR (y\AND -f(1)),
    $$
 that $\mathbf{W}$ and $\mathbf{V}(\gA)$ are independent, moreover $\mathbf{F}$ is an embedding. That $\mathbf{F}(\mathbf{W})$ does not have the CEP is clear from $\mathbf{1}\times \gA\in \mathbf{F}(\mathbf{W})$. Actually, no algebra of the form $\gB\times \gA$, where $\gB\in \mathbf{W}$ can have the CEP, by the Fraser-Horn property. That $\mathbf{F}(\mathbf{W})$ has $\mathbf{HS}=\mathbf{SH}$ follows from Lemma \labelcref{hs=sh} and the independence of the varieties.
\end{proof}

\begin{corollary}
    There are continuum many varieties of Boolean frames generated by algebras each of which has $\mathbf{HS}=\mathbf{SH}$, but lacking  CEP.
\end{corollary}

\section{CEP is not elementary}

 While certain identities (e.g. additivity) can force a Boolean frame to have the CEP, no identity can force an algebra not to have the CEP. The reason is that the trivial algebra satisfies every identity and has the CEP. There are, of course, first-order formulas that force an algebra not to have the CEP.  A straightforward example is any first-order formula that describes the isomorphism-type of a finite algebra that does not have the CEP. We show in this section that the dividing line between CEP and no-CEP (in the case of Boolean frames) cannot be first-order defined. Formally: there are elementarily equivalent Boolean frames $\gA$ and $\gB$ such that $\gA\in\CEP$, but $\gB\notin\CEP$. We reach this goal by constructing an algebra $\gA\in\CEP$ such that its ultrapower $\gA^{\omega}/U\notin\CEP$. In particular, this shows that the class $\CEP$ is not closed under ultrapowers. The example we construct is such that $\mathbf{HS}(\gA)=\mathbf{SH}(\gA)$, while $\mathbf{HS}(\gA^{\omega}/U)\neq\mathbf{SH}(\gA^{\omega}/U)$ showing that the property $\mathbf{SH}=\mathbf{HS}$, which is weaker than CEP, is also not preserved under ultrapowers.  \\

Take the powerset Boolean algebra 
$(\cP(\bbN), \AND, -, \emptyset)$ and define the
operator $f:\cP(\bbN)\to\cP(\bbN)$ by additively extending
$f(\emptyset)=\emptyset$, $f(\{0\}) = \bbN$ and $f(\{n\})=\{n-1\}$ for
$n\neq 0$. For $X\subseteq \bbN$ we have $f(X)=\bbN$ if and only if 
$0\in X$. Let $\gB = (\cP(\bbN), \AND, -, \emptyset, f)$ be 
the corresponding Boolean frame.

\begin{lemma}
    Every subalgebra of $\gB$ is simple.
\end{lemma}
\begin{proof}
    From any nonempty $a\in B$ by iteratively applying $f$ one reaches
    the top element $\bbN$. If $\equiv$ is a congruence of a subalgebra, and $a\equiv \emptyset$ for a nonempty element $a$ of the subalgebra, then $\emptyset\equiv f^{n}(a) = \bbN$ for a suitable finite $n$, yielding that $\equiv$ is the largest congruence.    
\end{proof}

\begin{lemma}\label{lem:nosimple}
    The ultrapower $\gB^{\omega}/U$ is not simple. ($U$ is a non-principal ultrafilter).
\end{lemma}
\begin{proof}
    Write $\gA = \gB^{\omega}/U$, and identify $\gB$  
    with its image in $\gA$ along the diagonal embedding. Elements of
    $B$ are referred to as standard, while those of $A\setminus B$ are
    non-standard. Let us call an element $a=\<a_i:i<\omega\>/U$
    \emph{finite} if $\{i: a_i$ is finite$\}\in U$. Note that $f$ is
    additive. Let
    \begin{align}\label{ideal}
        I = \big\{
            a\in A:\; a\text{ is finite, and if $b$ is standard and } 
            b\leq a,
            \text{ then } b=\emptyset
        \big\},
        \end{align}

    \noindent that is, $I$ consists of those non-standard finite elements which
    do not contain any standard element apart from the bottom of the Boolean algebra. These elements are such that no finite iteration 
    of $f$ gives the top element of the algebra. 
    $I$ is not empty, because 
    $\< \{i\}:i<\omega\>/U\in I$. We claim that $I$ is a congruential ideal.  

    Downward closedness of $I$ is immediate. If $x\OR y$ contains
    a nonzero standard element, then it contains a standard 
    atom $b$ as well. But for an atom $b$ we have $b\leq x\OR y$ 
    if and only if $b\leq x$ or $b\leq y$. Therefore, if $x,y\in I$, 
    then $x\OR y\in I$. 
    Next, take $x = \<x_i:i<\omega\>/U\in I$ with 
    $x_i\subseteq\bbN$ finite. It is immediate that $f(x)\in I$, and thus $I$ is closed under $f$.

    Finally, it is enough to prove that if $f$ is additive, then
    an $f$-closed ideal $I$ is congruential. Assume that $a\oplus b\in I$.
    We need $f(a)\oplus f(b)\in I$. As 
    \[a\oplus b = (a\AND -b)\OR(-a\AND b)\]
    and $I$ is downward closed, 
    $a\AND -b$, $-a\AND b\in I$. As $I$ is $f$-closed,
    $f(a\AND -b)$, $f(-a\AND b)\in I$,
    and then using additivity of $f$ and $\OR$-closedness of $I$ we get
    \[
        f(a\oplus b) = f((a\AND -b)\OR(-a\AND b))
        = f(a\AND -b)\OR f(-a\AND b)\in I\,.
    \]
    To complete the proof, by 
    appealing to downward closedness of $I$, it is enough to prove $f(a)\oplus f(b)\leq f(a\oplus b)$. In particular, it suffices to show
    \[
        f(a)\AND -f(b) \leq f(a\AND -b), 
        \text{ and }
        -f(a)\AND f(b) \leq f(-a\AND b).
    \]
    Recall that every additive function is monotonic, and thus 
    $f(x\AND y)\leq f(x)\AND f(y)$. Now, \vspace{-0.2cm}
    \begin{align*}
        f(a)\AND -f(b)
        &= \big(f(a\AND b)\OR f(a\AND -b)\big)\AND -f(b) \\
        &=  \big(f(a\AND b)\AND -f(b)\big)\OR
        \big(f(a\AND -b)\AND -f(b) \big)\\
        &=\0 \OR \big(f(a\AND -b)\AND -f(b) \big) \\
        &\leq f(a\AND -b). 
    \end{align*}
    The equality $-f(a)\AND f(b) \leq f(-a\AND b)$
    can be obtained by symmetry.
\end{proof}

So far we have an atomic Boolean frame $\gB$ whose subalgebras are simple, but its ultrapower is not simple. Next, we define the Boolean frame $\gA$
whose Boolean algebra reduct is that of $\gB\times \gB$, and the 
operation $f:(B\times B)\to(B\times B)$ is defined by
\[
    f( a,b ) = \begin{cases}
        (f(a), f(b)) & \text{ if } a=b,\\
        (\1,\1) & \text{ otherwise.}
    \end{cases}
\]
That is, $\gB$ diagonally embeds into $\gA$ as a Boolean frame, and $f$
maps every element not in the diagonal image of $\gB$ into the top element.

\begin{lemma}
    $\gA$ has the CEP, thus $\mathbf{HS}(\gA)=\mathbf{SH}(\gA)$.
\end{lemma}
\begin{proof}
    $\gA$ has the congruence extension property, because every 
    subalgebra of $\gA$ is simple: for any non-zero $(a,b)\in A$ there is 
    a finite $n$ such that $f^{n}(a,b)=(\1,\1)$. For $a\neq b$
    this $n=1$, and for $a=b$ this is by the construction of $\gB$. 
\end{proof}

\begin{lemma}
    $\mathbf{HS}(\gA^{\omega}/U)\neq \mathbf{SH}(\gA^\omega/U)$, thus $\gA^{\omega}/U$ does not have the CEP. ($U$ is a non-principal ultrafilter).
\end{lemma}
\begin{proof}
    $\gA$ is atomic, and every atom of $\gA$ is mapped to the top 
    element by $f$. This is a first-order property, hence it holds
    in $\gA^{\omega}/U$ as well. Also, $f^{\gA}(0)=0$, and so $\gA^{\omega}/U$ is simple, and $\mathbf{SH}(\gA^{\omega}/U)=\mathbf{S}(\gA^{\omega}/U,\mathbf{1})$.
    On the other hand, $\gA$ has a subalgebra isomorphic to $\gB$, and
    therefore $\gA^{\omega}/U$ has a subalgebra isomorphic to 
    $\gB^{\omega}/U$. By Lemma \labelcref{lem:nosimple} 
    this subalgebra is
    not simple, and thus there is a congruence of it which does not
    extend to a congruence of $\gA^{\omega}/U$. This directly shows the failure of CEP, but for the stronger statement we need a little more work. Let $\Theta$ be the congruence corresponding to the ideal $I$ from \labelcref{ideal}, that is 
    $$\<a,b\>\in \Theta \Leftrightarrow a\oplus b\in I.$$ Write $\gC = \gB^\omega/U$, then $\gC/\Theta\in \mathbf{HS}(\gA^\omega/U)$. Now consider the following elements in $\gB^\omega$:
    \[    
        e=\big\<\{2n+i: n\in \omega\}: i \in 2\mathbb{N}\setminus \{0\}\big\>,\quad o=\big\<\{2n+1+i: n\in \omega\}: i \in 2\mathbb{N}\big\>\,.
    \]
    Then, for
    \begin{align*}
        e/U&= \big\<\{2,4,6,\dots\},\{4,6,8,\dots\},\{6,8,10,\dots\},\dots\big\>/U, \text{ and }\\
        o/U&= \big\<\{1,3,5,\dots\},\{3,5,7,\dots\},\{5,7,9,\dots\},\dots\big\>/U
    \end{align*}
    we have $(e/U)/\Theta\neq (o/U)/\Theta$, because $\{n\in \omega: |e_n\oplus o_n| \text{ is infinite}\}\in U$. 
    It is easy to see that $f^\gC(e/U)= o/U$, and therefore
    $f^{\gC/\Theta}((e/U)/\Theta)= (o/U)/\Theta$. We also show $f^{\gC/\Theta}((o/U)/\Theta) = (e/U)/\Theta$. For, note first that 
    \[
        f^\gC(o/U) = \big\<\{0,2,4,\dots\},\{2,4,6,\dots\},\{4,6,8,\dots\},\dots\big\>/U,
    \]
    and for each 
    $n \in \omega$ we have $f^\gB(o_n)\oplus e_n =\{2n\}$. Therefore, 
    $\{n\in\omega:\; |f^\gB(o_n)\oplus e_n|$ is finite$\}\in U$. 
    It follows that $f^{\gC/\Theta}((o/U)/\Theta)= (f^\gC(o/U))/\Theta = (e/U)/\Theta$.
    Consequently, in $\gC/\Theta$ we have found distinct elements $a,b$, for which  $f(a) =b$ and $f(b) = a$, while in $\gA^\omega/U$ there are no such elements. This means $\gC/\Theta\not\in \mathbf{S}(\gA^\omega/U, \mathbf{1})$, and therefore  $\mathbf{HS}(\gA^{\omega}/U)\neq \mathbf{SH}(\gA^\omega/U)$.
\end{proof}

\noindent Let $\CEP$ be the class of Boolean frames with the congruence extension
property.

\begin{theorem}
    Neither $\CEP$ nor its complement are elementary.
\end{theorem}
\begin{proof}
    $\gA\in \CEP$ is elementarily equivalent to $\gA^{\omega}/U\notin\CEP$. 
\end{proof}

That the complement of $\CEP$ is closed under ultrapowers can be seen as follows. Take any algebra $\gA\notin\CEP$ and let $\gB\subseteq\gA$ be a subalgebra with $\Theta\in\Con(\gB)$ such that $\Theta$ does not extend to a congruence of $\gA$. As $\gA$ and $\gB$ can be diagonally embedded into any ultrapower $\gA^I/U$, if $\gA^I/U$ had the congruence extension property, then $\Theta$ would extend to a congruence $\Psi\in\Con(\gA^I/U)$. But then $\Psi\cap (A\times A)$ would be a congruence of $\gA$ extending $\Theta$, contradicting the assumptions. We do not know whether or not the complement of $\CEP$ is closed under \emph{ultraproducts}. 

In the class of Boolean algebras with infinitely many unary operations it is easy to construct algebras $\gA_i$ without the CEP such that their ultraproduct $\Pi_{i\in I}\gA_i/U$ has the CEP. For, let $\gA_i = (A, f_0^i, \ldots, f_n^i, \ldots)$ be such that all the $f_n^i$ are additive except for $f_i^i$ which forces the failure of CEP in $\gA_i$. Then in $\Pi_{i\in I}\gA_i/U$ every operation is additive, ensuring that the ultraproduct has the CEP.

\begin{problem}
    Are there Boolean frames $\gA_i$ (or Boolean algebras with \emph{finitely} many extra unary functions) such that every $\gA_i$ fails to have the CEP, while an ultraproduct $\Pi_{i\in I}\gA_i/U$ admits CEP?
\end{problem}

In order to deal with the congruence extension property of an ultraproduct some kind of a uniform bound on the generation of congruences would be useful. The closest we could get to this is the result below, which might be interesting on its own: For the CEP of a Boolean frame it is enough to check whether principal congruences of $2$-generated subalgebras extend to the whole algebra.

\begin{theorem}\label{thm:3gen}
    A Boolean frame $\gA$ has the CEP if and only if every principal congruence of every $2$-generated subalgebra extends to a congruence of $\gA$.
\end{theorem}
\begin{proof}
    We split the proof into several lemmas.
    \begin{lemma}\label{lem:pconjoin}
        For a Boolean frame $\gA$ and principal congruences $\Theta, \Psi\in \Con(\gA)$ the join
        $\Theta\lor\Psi\in\Con(\gA)$ is a principal congruence.
    \end{lemma}
    \begin{proof}
        Since for a principal congruence $\Theta(a,b)$ we have
        $\Theta(a,b)=\Theta(a\oplus b, \0)$, it is enough to prove
        \[
            \Theta(x,\0)\lor\Theta(y,\0) = \Theta(x\OR y, \0)\,.
        \]
        This follows from the Boolean structure:
        $(x\OR y, \0)\in \Theta(x\OR y, \0)$, and then using 
        $x,y\leq x\OR y$ we also have $(x, \0), (y,\0)\in \Theta(x\OR y, \0)$, and thus $\Theta(x,\0)\lor\Theta(y,\0) \subseteq \Theta(x\OR y, \0)$. On the other hand, 
        $(x,\0), (y,\0)\in \Theta(x,\0)\lor\Theta(y,\0)$, and
        thus $(x\OR y,\0)\in \Theta(x,\0)\lor\Theta(y,\0)$, 
        showing $\Theta(x\OR y, \0)\subseteq \Theta(x,\0)\lor\Theta(y,\0)$.
    \end{proof}

    \begin{lemma}\label{lem:punion}
        Let $\gB\subseteq\gA$ be Boolean frames and let $\Theta\in\Con(\gB)$. Then
        \[
            \bigcup_{(x,y)\in\Theta}\Theta^{\gA}(x,y)
        \]
        is a congruence of $\gA$.
    \end{lemma}
    \begin{proof}
        Take $(x_1,y_1),(x_2,y_2)\in\Theta$, and put
        \[
            u=(x_1\oplus y_1)\OR (x_2\oplus y_2).
        \]
        Since $(x_i, y_i)\in\Theta$, we have $(x_i\oplus y_i, 0)\in \Theta$, and hence $(u,0)\in\Theta$. By Lemma \labelcref{lem:pconjoin},
        \[
            \Theta^{\gA}(x_1,y_1)\vee
            \Theta^{\gA}(x_2,y_2)=
            \Theta^{\gA}(u,0).
        \]
        Since$(u,0)\in\Theta$, the right-hand side belongs to the family $\{ \Theta^{\gA}(x,y):\; (x,y)\in\Theta\}$. Thus this family is nonempty and closed under binary joins, and so it is directed. For any algebraic closure operation (in particular for the closure operation of taking generated congruences) 
        the union of an upward directed set of closed sets is closed (see \cite{BurrisSankappanavar1981}). It follows that 
        \[
            \bigcup_{(x,y)\in\Theta}\Theta^{\gA}(x,y) = 
            \bigvee_{(x,y)\in\Theta}\Theta^{\gA}(x,y)
            \;\in\;\Con(\gA)\,.
        \]
    \end{proof}

    \noindent Let $\gB$ be a subalgebra of the Boolean frame $\gA$, and let $\Theta^{\gB}(x,y)\in\Con(\gB)$ be a principal congruence. If $\Theta^{\gB}(x,y)$ extends to a congruence of $\gA$, then $\Theta^{\gA}(x,y)$ is such an extension, that is $\Theta^{\gB}(x,y) = \Theta^{\gA}(x,y)\cap (B\times B)$. An algebra has the principal congruence extension property (PCEP) if every principal congruence of each of its subalgebras extends to the entire algebra.

    \begin{lemma}\label{lem:pcepeleg}
        A Boolean frame $\gA$ has the CEP if and only if it has the PCEP.
    \end{lemma}
    \begin{proof}
        That the CEP implies the PCEP is obvious. For the other direction assume that $\gA$ has the PCEP. Take a subalgebra $\gB\subseteq\gA$ and a congruence $\Theta^\gB$. Then 
        \begin{align*}
        \Theta^{\gB} &= \bigcup_{(x,y)\in \Theta}\Theta^{\gB}(x,y) \overset{\text{PCEP}}{=} \bigcup_{(x,y)\in\Theta} \big( \Theta^{\gA}(x,y)\cap (B\times B) \big) \\
        &= \Big(\bigcup_{(x,y)\in\Theta}\Theta^{\gA}(x,y)\Big)\cap (B\times B)\,.
        \end{align*}
        This latter union is a congruence by Lemma \labelcref{lem:punion}.
    \end{proof}

    Getting back to the proof of Theorem \labelcref{thm:3gen}, let $\mathfrak B\leq \mathfrak A$, and let $a,b\in B$.
    We must show that the principal congruence 
    $\operatorname{Cg}^{\mathfrak B}(a,b)$
    extends to a congruence of $\mathfrak A$.  Suppose, towards a contradiction, that $\operatorname{Cg}^{\mathfrak B}(a,b)$ does not extend to $\mathfrak A$. Then
    \[
        \operatorname{Cg}^{\mathfrak A}(a,b)\cap (B\times B)\neq \operatorname{Cg}^{\mathfrak B}(a,b),
    \]
    so there exist $c,d\in B$ such that
    \[
        (c,d)\in \operatorname{Cg}^{\mathfrak A}(a,b)
        \qquad\text{but}\qquad
        (c,d)\notin \operatorname{Cg}^{\mathfrak B}(a,b),
    \]  
    equivalently,
    \begin{align}
        (c\oplus d, \0)\in 
        \operatorname{Cg}^{\mathfrak A}(a\oplus b,\0)
        \qquad\text{but}\qquad
        (c\oplus d,\0)\notin \operatorname{Cg}^{\mathfrak B}(a\oplus b,\0).\label{eq:con}
    \end{align}
    Consider now the subalgebra 
    $\mathfrak C=\operatorname{Sg}^{\gB}(\{a\oplus b, c\oplus d\})$.
    It is $2$-generated, hence by assumption the principal congruence
    $\operatorname{Cg}^{\mathfrak C}(a\oplus b, \0)$
    extends to some congruence $\Phi\in \operatorname{Con}(\mathfrak A)$.
    Since $(a\oplus b, \0)\in \Phi$, by the minimality of the principal congruence generated by $(a\oplus b,\0)$ in 
    $\mathfrak A$,
    $\operatorname{Cg}^{\mathfrak A}(a\oplus b, \0)\subseteq \Phi$.
    As $(c\oplus d,\0)\in \operatorname{Cg}^{\mathfrak A}(a\oplus b,\0)$, we get
    $(c\oplus d,\0)\in \Phi$. Now $c\oplus d\in C$, and $\Phi$ extends 
    $\operatorname{Cg}^{\mathfrak C}(a\oplus b,\0)$, so
    \[
        (c\oplus d,\0)\in \Phi\cap (C\times C)=\operatorname{Cg}^{\mathfrak C}(a\oplus b,\0).
    \]
    On the other hand,
    $\operatorname{Cg}^{\mathfrak C}(a\oplus b,\0)\subseteq \operatorname{Cg}^{\mathfrak B}(a\oplus b,\0)$,
    because the restriction of 
    $\operatorname{Cg}^{\mathfrak B}(a \oplus b,\0)$ to $C$ is a congruence of $\mathfrak C$ containing $(a \oplus b, \0)$.
    Therefore $(c\oplus d,\0)\in \operatorname{Cg}^{\mathfrak B}(a \oplus b,\0)$,
    contradicting \labelcref{eq:con}.
\end{proof}

\section*{Acknowledgements}
Research supported in part by the Hungarian National Research, Development and Innovation Office, contract number: K-134275, and K-152165, and by the project no. 2022/47/B/HS1/01581 of the National Science Centre, Poland.

\end{document}